\documentclass[amscd, amssymb, reqno, verbatim, 11pt]{amsart}
\usepackage{amssymb, stmaryrd, graphicx, caption2, color}

\def\be{\begin{eqnarray}}
\def\ee{\end{eqnarray}}
\def\bea{\begin{eqnarray*}}
\def\eea{\end{eqnarray*}}
\def\na{\nabla}

\def\d{\delta}

\def\vp{\varphi}

\def\a{\alpha}
\def\n{\nabla}

\def\o{\omega}
\def\tr{\Delta}
\def\dis{\displaystyle}

\newtheorem{defi}{Definition}[section]

\newtheorem{lem}[defi]{Lemma}
\newtheorem{remark}[defi]{Remark}
\newtheorem{thm}[defi]{Theorem}
\newtheorem{prop}[defi]{Proposition}
\newtheorem{cor}[defi]{Corollary}

\numberwithin{equation}{section}

\begin{document}

\title[ On the structure of  Closed generalized Einstein manifolds]
{On the structure of  closed generalized Einstein manifolds and  rigidity properties}

\author[S. Hwang]{Seungsu Hwang}

\author[M. Santos]{Marcio Santos}

\author[G. Yun]{Gabjin Yun$^*$}


\thanks{$*$ Corresponding author}
\thanks{The first author was supported by the Basic Science Research Program through the National Research Foundation of Korea(NRF)  funded by the Ministry of Education(NRF-2018R1D1A1B05042186), and the third author by the Ministry of Education(RS-2024-00334917).}


\begin{abstract}
In this paper, we show that a closed $n$-dimensional generalized $(\lambda, n+m)$-Einstein  manifold  of constant scalar curvature  is isometric to either  a sphere  ${\Bbb S}^n$, or a product ${\Bbb S}^{1} \times \Sigma^{n-1}$ of a circle with an $(n-1)$-dimensional Einstein manifold of positive Ricci curvature,  up to finite cover and rescaling 
when $\omega = df  \wedge i_{\na f}z = 0$.
Furthermore, if we assume $(M, g)$ has positive isotropic curvature, $M$  must be isometric to either
  a sphere  ${\Bbb S}^n$, or a product ${\Bbb S}^{1} \times {\Bbb S}^{n-1}$ of a circle with an  $(n-1)$-sphere.
 \end{abstract}

\maketitle

\setlength{\baselineskip}{15pt}

\section{Introduction}

A closed generalized $(\lambda, n+m)$-Einstein manifold is a triple $(M^n, g, f)$, where $(M^n, g)$ is
a closed $n$-dimensional Riemannian manifold and $f$ is a smooth function on $M$
satisfying
\be
Ddf = \frac{f}{m} ({\rm Ric} - \lambda g).\label{eqn1}
\ee
 Furthermore, $\lambda$ is a smooth function on $M$ and $m$ is a positive real number.
In particular,  $(M, g, f)$ will be called a $(\lambda, n+m)$-Einstein manifold if $\lambda$ is constant.

The  motivation to approach this type of manifolds is studying Einstein manifolds that have a structure of warped product. In fact, if $m>1$ is an integer and $f>0$, it is known \cite{kim2} that
$(M^n,g,f)$ is a $(\lambda,n+m)$-Einstein manifold  if and only if there is a smooth $(n+m)$-dimensional warped product Einstein manifold having $M$ as the base space.  
We also observe that if we define the function $\phi$  by $e^{-\frac{\phi}{m}}=f$,  the equation \eqref{eqn1} becomes
$$
{\rm Ric}_\phi^m:={\rm Ric}+Dd \phi-\frac{d\phi\otimes d\phi}{m}=\lambda g.
$$
Here, the tensor ${\rm Ric}_\phi^m$  is the well-known $m$-Bakry-\'Emery Ricci tensor.  Taking $m\to\infty$, we get the gradient (generalized) Ricci soliton equation
$${\rm Ric}+Dd\phi=\lambda g.$$

A  generalized $(\lambda, n+m)$-Einstein manifold $(M^n, g, f)$ is called {\it trivial} 
if the potential function $f$ is constant, and in this case $(M, g)$ is Einstein.  
We point out that,  for a closed  $(\lambda, n+m)$-Einstein manifold $(M^n, g, f)$ 
having constant scalar curvature, if 
 $f$ does not change sign on $M$ and $\lambda$ is constant,  it follows from the maximum
  principle that $(M^n, g, f)$ is trivial.  On the other hand,  if $\lambda$ is non-constant, 
  a closed generalized $(\lambda, n+m)$-Einstein manifold with constant scalar curvature 
  and $f>0$ is Einstein and isometric to a sphere, see  \cite{b-g}.
Motivated by these results,  from now on, we always assume that 
$$
\min_M f :=a<0 < b:=\max_M f
$$
so that $f^{-1}(0)$ is a non-empty set whenever we consider closed generalized 
$(\lambda, n+m)$-Einstein manifolds.

In this paper, we consider closed generalized $(\lambda, n+m)$-Einstein manifolds $(M, g, f)$ satisfying 
$\o:=df \wedge i_{\n f}z_g = 0$, where $z_g$ is the traceless Ricci curvature tensor defined by $z_g = {\rm Ric}- \frac{Scal_g}{n}g$ and  $i_{\n f}z$ denotes the interior product. We can see that the vanishing of $\o$ is equivalent to $z_g(\n f, X) = 0$ for any vector field $X$ orthogonal to $\n f$ since vacuum static space have constant scalar curvature.

Our main result is the following.

\begin{thm}
Let $(M^n, g, f)$ be a closed generalized $(\lambda, n+m)$-Einstein manifold with $f^{-1}(0) \ne \emptyset.$
Suppose that $(M, g)$ has  constant scalar curvature and $\o=df \wedge i_{\n f}z = 0$.  Then, up to finite cover and rescaling, we have the following.
\begin{itemize}
\item[(1)]
If $f^{-1}(0)$ is connected, then  $(M, g)$ is isometric to a sphere ${\Bbb S}^n$.
\item[(2)] If $f^{-1}(0)$ is disconnected, then $f^{-1}(0)$ has exactly two connected components and $(M, g)$ is isometric  to the product ${\Bbb S}^1 \times {\Sigma}^{n-1}$. Here $\Sigma$ is an $(n-1)$-dimensional Einstein manifold of positive Ricci curvature.
\end{itemize} 
\end{thm}

Under  a little stronger condition than $\o=df \wedge i_{\n f}z = 0$, we have the following.

\begin{cor}
Let $(M^n, g, f)$ be a closed generalized $(\lambda, n+m)$-Einstein manifold of constant scalar curvature
with $f^{-1}(0) \ne \emptyset.$
If $z(\n f, \n f) = 0$ on $M$, then $(M, g)$ is isometric to a sphere ${\Bbb S}^n$.

\end{cor}

For a  geometric equation like gradient Ricci soliton equation, many authors studied rigidity result under the hypothesis of vanishing 
Ricci curvature or trace-less Ricci curvature in the direction of the gradient of potential function.
For example, in \cite{p-w},  Petersen and Wylie proved that any compact gradient Ricci soliton $(M^n, g)$ satisfying
${\rm Ric} + Ddf = \lambda g$  is rigid if the scalar curvature is constant and ${\rm Ric}(\n f, \n f) = 0$.
Recently, in \cite{b-b-b-v}, Baltazar et al. proved rigiditys result for static vacuum spaces with zero radially Weyl curvature which means 
$i_{\n f}\mathcal W = \mathcal (\n f, \cdot, \cdot, \cdot) = 0$, where $\mathcal W$ denotes the Weyl curvature tensor.

\vspace{.1in}
Next, we would like to point out that, when $(M^n, g, f)$ is a closed $(\lambda, n+m)$-Einstein manifold, that is, $\lambda$ is a constant function, the condition $\omega=0$ is equivalent to constant scalar curvature (see Lemma \ref{constantscalar}). We can also easily show that there are no critical points of $f$ on the set $f^{-1}(0)$ and each connected component of $f^{-1}(0)$ is totally geodesic (see Section 2). 
We can also show that the  vanishing of $\o =0$  implies 
the potential function $f$ does not have critical points except at minimum and  maximum points.

Based on these properties on the potential function  for  a closed $(\lambda, n+m)$-Einstein manifold, we can also show the following.

\begin{thm}
Let $(M^n, g, f)$ be a closed generalized $(\lambda, n+m)$-Einstein manifold with $f^{-1}(0) \ne \emptyset.$
Suppose that $(M, g)$ has  constant scalar curvature and $\o=df \wedge i_{\n f}z = 0$. If the minimum set of $f$ is a hypersurface in $M$, then $(M^n, g)$ has harmonic Weyl curvature and  Bach-flat.

\end{thm}


 If, furthermore,  $(M, g)$  has positive isotropic curvature, then $\Sigma$ can be shown to be a round sphere ${\Bbb S}^{n-1}$.

\begin{thm}
Let $(M^n, g, f), n \ge 4,$ be a closed generalized $(\lambda, n+m)$-Einstein manifold with $f^{-1}(0) \ne \emptyset.$
Suppose that $(M, g)$ has   positive isotropic curvature and  constant scalar curvature with
$\o=df \wedge i_{\n f}z = 0$.  If $f^{-1}(0)$ is disconnected, then, up to finite cover and rescaling,  $(M, g)$ is isometric  to the product ${\Bbb S}^1 \times {\Bbb S}^{n-1}$. 
\end{thm}

We recall that a Riemannian manifold $M$ has positive isotropic curvature  (PIC in short) if and only if,  for every orthonormal four-frame $ \{e_1,e_2,e_3,e_4\}$,  we have that 
$$
R_{1313}+R_{1414}+R_{2323}+R_{2424}-2R_{1234}>0.
$$
The positive isotropic curvature was first introduced by Micalleff and Moore \cite{mm88}  in consideration of the second variation of energy of maps from surfaces into $M.$ 
It is easy to see that a standard sphere ${\Bbb S}^n$ and a product
${\Bbb S}^{n-1}\times {\Bbb S}^1$ or ${\Bbb S}^{n-1}\times {\Bbb R}$ have PIC.
It is also well-known that if the sectional curvature of a Riemannian manifold  is 
pointwise strictly quater-pinched, then it has PIC, and the connected sum of manifolds with PIC
admits a PIC metric (see \cite{mm88} and \cite{m-w}). 
This condition was also used by Brendle and Schoen  to prove the celebrated differentiable pointwise $1/4$-pinching sphere theorem, see \cite{BS}.

The paper is organized as follows. In Section 2, we describe basic properties on generalized $(\lambda, n+m)$-Einstein manifolds and show a rigidity result of them when $\lambda =1$.
In Section 3, we introduce a $3$-tensor and derive its properties and relations with the Cotton tensor and Bach tensor. We also show some useful facts on a special $2$-form $\omega:= df\wedge i_{\n f}z$ for a generalized $(\lambda, n+m)$-Einstein manifold $(M^n, g, f)$.
In Section 4, we prove  the second part of Theorem 1.1 and other Theorems (Corollary 1.2, Theorem 1.3 and Theorem 1.4) mentioned above. 
The first part of Theorem 1.1 is proved in  Section 5, and in Section 6, we add some remarks on  the relation of scalar curvature with $\lambda$ for generalized $(\lambda, n+m)$-Einstein manifolds.

\vskip .5pc

\noindent
{\bf Convention and  Notations: }  Basically, we follow curvature conventions and operator conventions in \cite{Be} except only  the Laplace operator. Hereafter, for convenience and simplicity, we denote curvatures 
${\rm Ric}_g,  z_g,   {\rm Scal}_g$, and   the Hessian and Laplacian of $f$, $D_gdf, \Delta_g$  by
$r, z, s$, and $Ddf, \Delta$, respectively, if there are no ambiguities.
We also use the notation $\langle \,\, ,\,\, \rangle$ for metric $g$ or inner product induced by $g$ on tensor spaces.

\section{Basic Properties}

In this section, we give some basis properties on closed generalized $(\lambda, n+m)$-Einstein manifolds. Before doing this, we first enumerate  equivalent equations 
to (\ref{eqn1}) and basic identities which are used later.

\vspace{.13in}
\noindent
{\bf Basic Identities:} 
\begin{itemize}
\item[(i)] Taking the trace of (\ref{eqn1}), we have
\be
\Delta f = \frac{s - n\lambda}{m} f.\label{eqn2}
\ee
\item[(ii)] Denoting $\dis{\mathring{Ddf} = Ddf- \frac{\Delta f}{n}g}$, the equation (\ref{eqn1}) is reduced to
\bea
f z = m \mathring {Ddf}.\label{eqn3}
\eea
\item[(iii)] Since $\dis{\frac{\lambda f}{m} = \frac{1}{n}\left(\frac{sf}{m} - \Delta f \right)}$ in (\ref{eqn2}), we can express
 (\ref{eqn1}) as
\bea
Ddf = \frac{f}{m} r - \frac{1}{n} \left(\frac{sf}{m} - \Delta f \right) g. \label{eqn4}
\eea
\item[(iv)] Using $z = r - \frac{s}{n}g$, the equation (\ref{eqn1}) becomes
\be
fz = m Ddf + \left(\lambda - \frac{s}{n}\right) f g = m Ddf - \frac{m}{n} (\Delta f) g.\label{eqn5}
\ee
\item[(v)] Letting $\mu:= \lambda f - \frac{s}{n-1} f$, from (\ref{eqn5}), we have
\be
fz =mDdf + \left(\frac{sf}{n(n-1)} + \mu\right)g\label{eqn16}
\ee
and, by taking the trace of (\ref{eqn16}), we obtain
\be
\Delta f = - \frac{s}{m(n-1)}f -\frac n{m}\mu.\label{eqn2025-3-16-1}
\ee
\end{itemize}

\begin{lem}\label{lem2021-5-15-1}
Let $(M^n, g, f)$ be a  closed generalized $(\lambda, n+m)$-Einstein manifold
with $f^{-1}(0) \ne \emptyset$.
 Then,  the set $\{f>0\}\cap \{\lambda \ge \frac{s}{n}\}$ is nonempty.
\end{lem}
\begin{proof}
Suppose, on the contrary, that $\{f>0\}\cap \{\lambda \ge \frac{s}{n}\} = \emptyset$.
 Multiplying (\ref{eqn2}) by $f$ and integrating it over  the set $f\ge 0$, we obtain
$$
\int_{f\ge 0} |\n f|^2 =  \frac{n}{m} \int_M \left(\lambda - \frac{s}{n}\right) f^2 \le 0,
$$
which shows that $\n f = 0$ and so $f=0$ on the set $\{f >0\}$, a contradiction.
\end{proof}
The same argument in the proof of Lemma~\ref{lem2021-5-15-1} shows that
$$
\{f<0\}\cap \{\lambda \ge \frac{s}{n}\}\ne \emptyset.
$$

\begin{lem} \label{prop2021-5-15-2}
Let $(M^n, g, f)$ be a closed generalized $(\lambda, n+m)$-Einstein manifold with $f^{-1}(0) \ne \emptyset$.
 Then, there are no critical points of $f$ on the set $f^{-1}(0)$.
 \label{prop1}
\end{lem}
\begin{proof}
Suppose that there is a critical point  $p\in f^{-1}(0)$ of $f$. Let $\gamma$ be a unit-speed
geodesic starting at $p$ and define $\vp(t)=f\circ \gamma(t)$.
From (\ref{eqn5}), we have
$$
 \vp''(t)= \frac{1}{m} \left[z (\gamma'(t), \gamma'(t))  - \left(\lambda - \frac {s}{n}\right)\right] \vp(t)
$$
with $\vp(0)=0$ and $\vp'(0)=0$. So, it follows from the uniqueness of ODE solution that $\vp$ vanishes identically, which is a contradiction.
\end{proof}

We can easily deduce  from Lemma~\ref{prop2021-5-15-2} that any connected component of the set $f^{-1}(0)$ is a hypersurface of $M$.
Moreover, it is easy to see that  $|\n f|^2$ is a positive constant on each connected component of  $f^{-1}(0)$. 
In fact, if $X$ is  tangent  to (a connected component of) $f^{-1}(0)$, by (\ref{eqn5}), we have
$X(|\n f|^2) = 0$.  Furthermore, we have the following.

\begin{lem}\label{lem202-5-15-3}
Let $(M^n, g, f)$ be a closed generalized $(\lambda, n+m)$-Einstein manifold
with $f^{-1}(0) \ne \emptyset$.
Then, any connected component of $f^{-1}(0)$ is a totally geodesic hypersurface in $M$.
\end{lem}
\begin{proof}
We can take  $N = \frac{\n f}{|\n f|}$ as a unit normal vector field on (a component of)
$f^{-1}(0)$ by Lemma~\ref{prop2021-5-15-2}. 
Choosing a local frame $\{N, e_2, e_3, \cdots, e_{n}\}$ 
so that $\{e_2, e_3, \cdots, e_{n}\}$ are tangent to $f^{-1}(0)$, we have
$e_i (|\n f|) = e_i \left(\frac{1}{|\n f|}\right) = 0$ on the set $f^{-1}(0)$, which implies 
$D_{e_i}N = 0$.
\end{proof}

It is well-known that the following identities  hold for a Riemannian manifold  in general.
\be
\d r = - \frac{1}{2} ds \quad \mbox{and}\quad \d z = - \frac{n-2}{2n}ds.\label{eqn2020-12-2-6}
\ee
Here $\d = {\rm -div}$ denotes the negative divergence operator.

\begin{lem}\label{lem2021-4-19-100}
Let $(M^n, g, f)$ be a generalized $(\lambda, n+m)$-Einstein manifold.  
Let $\mu:= \lambda f - \frac{s}{n-1} f$. Then we have
\be
(m-1)i_{\n f}z  + \frac{m-1}{n} s df = \frac{1}{2}f ds + (n-1) d\mu.\label{eqn2021-4-19-15}
\ee
\end{lem}
\begin{proof}
Taking the divergence operator $\d$  of (\ref{eqn16}), we obtain
\be
-i_{\n f}z + f \d z = m(-i_{\n f}r - d\Delta f) - \frac{s}{n(n-1)}df - \frac{f}{n(n-1)}ds - d\mu.
\label{eqn2021-4-19-11}
\ee
Recall that $\d z = - \frac{n-2}{2n} ds$ and $i_{\n f}r = i_{\n f}z + \frac{s}{n}df$  by (\ref{eqn2020-12-2-6}).
Substituting these into (\ref{eqn2021-4-19-11}) and using (\ref{eqn2025-3-16-1}), we obtain our conclusion.

\end{proof}

When $m=1$, for a closed  generalized $(\lambda, n+1)$-Einstein manifolds, we have the following.
 we have the following.

\begin{cor}\label{cor2021-4-19-12}
Let  $(M^n, g, f)$ be a closed generalized $(\lambda, n+1)$-Einstein manifold with $f^{-1}(0)\ne \emptyset$. 
If the scalar curvature $s$ is constant, then $s$ should be positive and $\lambda$ is also  a positive constant. In fact,  we have
$$
\lambda = \frac{s}{n-1}.
$$
\end{cor}
\begin{proof}
Suppose that $s$ is constant. Letting $\mu= \lambda f - \frac{s}{n-1} f$, it follows from Lemma~\ref{lem2021-4-19-100} that $\mu$ is constant. Since $f^{-1}(0)\ne \emptyset$, we have $\mu = 0$ and so $\lambda  = \frac{s}{n-1}$, which shows that $\lambda $ is constant. 
By (\ref{eqn2}), we have
$$
\Delta f = sf - n \lambda f  = - \frac{s}{n-1} f.
$$
Hence, it is easy to see that  $s$ should be positive.
\end{proof}

\begin{thm}
Let $(M^n, g, f)$ be a closed generalized $(\lambda, n+1)$-Einstein manifold with $f^{-1}(0)\ne \emptyset$. If $(M, g)$ has constant  scalar curvature and $\o=df\wedge i_{\n f}z = 0$, then, up to finite cover and rescaling, $M$ is isometric to a sphere ${\Bbb S}^n$ or the product ${\Bbb S}^1 \times {\Sigma}^{n-1}$ of a circle and an $(n-1)$-dimensional Einstein manifold $\Sigma$ of positive Ricci curvature. 
\end{thm}
\begin{proof}
From Corollary~\ref{cor2021-4-19-12}, we have $\lambda = \frac{s}{n-1}$ and so  the equation (\ref{eqn5}) is reduced to
$$
fz = Ddf + \frac{s}{n(n-1)},
$$
which is exactly a  static vacuum space. Applying one of main results in \cite{h-y5}, we obtain our conclusion.
\end{proof}

In case of positive isotropic curvature, we have the following.

\begin{cor}
 Let $(M^n, g, f), n \ge 4, $ be a closed generalized $(\lambda, n+1)$-Einstein manifold of constant scalar curvature with $f^{-1}(0)\ne \emptyset$. Suppose that $(M, g)$ has PIC and $\o = 0$. Then, up to finite cover and rescaling, $M$ is isometric to a sphere ${\Bbb S}^n$ or the product ${\Bbb S}^1 \times {\Bbb S}^{n-1}$.
\end{cor}
\begin{proof}
In case of product $M= {\Bbb S}^1 \times \Sigma^{n-1}$,  $\Sigma$ is homeomorphic to ${\Bbb S}^{n-1}$ by PIC condition 
(\cite{mm88}). Applying a result in \cite{BS}, it must be isometric to a round sphere ${\Bbb S}^{n-1}$ for $n \ge 5$. For $n=4$,   it is easy to see that $\Sigma = {\Bbb S}^{n-1}$ since $\Sigma$ is an Einstein manifold of positive Ricci curvature. 
\end{proof}

From now on ,  we assume  $m>1$ for generalized $(\lambda, n+m)$-Einstein manifolds $(M^n, g, f)$ throughout the paper.
For the case $m >1$, we can also show that the scalar curvature is positive as Corollary~\ref{cor2021-4-19-12} if it is constant.

\begin{prop}\label{prop2021-4-19-18}
Let $m>1$ be an integer. Let $(M^n, g, f)$ be a closed generalized 
$(\lambda, n+m)$-Einstein manifold with $f^{-1}(0)\ne \emptyset$. If  the scalar curvature $s$ is constant, then $s$ should be positive.
\end{prop}
\begin{proof}
Letting $\mu= \lambda f - \frac{s}{n-1} f$, it follows from Lemma~\ref{lem2021-4-19-100} that
$$
(m-1)i_{\n f}z + \frac{m-1}{n}s df = (n-1)d \mu.
$$
Taking the divergence operator $\d$ of this equation,  we have
$$
(m-1)\d i_{\n f}z - \frac{m-1}{n}s \Delta f = -(n-1)\Delta \mu.
$$
Since $\d i_{\n f}z = - \frac{n-2}{2n}\langle \n s, \n f\rangle - \frac{f}{m}|z|^2 
= - \frac{f}{m}|z|^2$, we obtain
\bea
\Delta \left[(n-1)\mu - \frac{m-1}{n}sf\right] = \frac{m-1}{m}f|z|^2.\label{eqn2021-5-15-6}
\eea
Applying the maximal principle to the function $(n-1)\mu - \frac{m-1}{n}sf$ on the set $f \ge 0$,  we have
$$
(n-1)\mu - \frac{m-1}{n}sf  \le 0
$$
on the set $f \ge 0$ since  $(n-1)\mu - \frac{m-1}{n}sf =0$ on the set $f^{-1}(0)$.  
Substituting $\mu= \lambda f - \frac{s}{n-1} f$ into this inequality, we have
\bea
n \lambda \le \frac{m+n-1}{n-1} s = s + \frac{m}{n-1}s\label{eqn2021-5-15-5}
\eea
on the set $f \ge 0$. 
Since $\{f>0\}\cap \{n\lambda >s\} \ne \emptyset$ by Lemma~\ref{lem2021-5-15-1}  and
$s$ is constant, we must have $s>0$.

\end{proof}

\section{Tensorial Properties}

In this section, we introduce a $3$-tensor $T$ and enumerate some relations $T$ with Cotton tensor and Weyl tensor 
for generalized $(\lambda, n+m)$-Einstein manifolds $(M^n, g, f)$. 
After doing this, we investigate basic properties for the $2$-form
$\omega:= df \wedge i_{\n f}z$ which plays an important role in proving our main results.

Let $(M^n, g)$ be a Riemannian manifold of dimension $n$ with the Levi-Civita
 connection $D$, and let $h$ be a symmetric $2$-tensor on $M$. The
differential $d^Dh$ is defined by 
$$ 
d^Dh(X, Y, Z) = D_Xh(Y, Z) - D_Yh(X, Z) 
$$
 for any vectors $X, Y$ and $Z$.
 
\begin{defi}
 The {Cotton tensor}  $C \in\Gamma(\Lambda^2 M \otimes T^*M)$ is defined by 
 \bea
 C = d^D \left({\rm Ric} - \frac{s}{2(n-1)} g\right) 
 = d^D{r} - \frac{1}{2(n-1)} ds \wedge g, \label{eqn2017-4-1-1}
 \eea 
 where $ds \wedge g$ is defined by $ds\wedge g(X, Y, Z) = ds(X)g(Y, Z) - ds(Y)g(X, Z)$ for vectors 
 $X, Y, Z$. 
\end{defi}

Using the trace-less Ricci tensor  $z = r-\frac{s}{n}g$, the Cotton tensor can be written as
 \be
 C = d^D z + \frac{n-2}{2n(n-1)} ds \wedge g.\label{eqn2020-11-12-2}
 \ee
Introducing a local orthonormal  frame $\{e_i\}$ and denoting $C_{ijk} = C(e_i, e_j, e_k)$, we have
\bea
\langle \d C, z\rangle &=& - C_{ijk;i}z_{jk} = - (C_{ijk}z_{jk})_{;i} + C_{ijk}z_{jk;i}\\
&=&
 - (C_{ijk}z_{jk})_{;i} + \frac{1}{2}|C|^2,
 \eea
 where the semi-colon denotes  covariant derivative. It is also easy to see that the cyclic summation of indices in $C$ vanishes: $C_{ijk}+C_{jki}+C_{kij} = 0$, and a trace of $C$ in any two summands also vanishes.

Related to the Cotton tensor $C$, the divergence of Weyl tensor $\mathcal W$  is given by
\be
\d \mathcal W = - \frac{n-3}{n-2}d^D \left({\rm Ric} - \frac{s}{2(n-1)} g\right) 
= - \frac{n-3}{n-2} C\label{eqn2016-12-3-16} 
\ee
under the following identification
$$ 
\Gamma(T^*M\otimes \Lambda^2M) \equiv \Gamma(\Lambda^2M \otimes T^*M). 
$$

Next, for a  generalized $(\lambda, n+m)$-Einstein manifold $(M^n, g, f)$, we define the tensor $T$ as
\bea
T = \frac{1}{n-2} df \wedge z + \frac{1}{(n-1)(n-2)} i_{\n f}z \wedge g,\label{eqn6}
\eea
where $i_{\n f}z$ denotes the interior product given by
$i_{\n f}z(X) = z(\n f, X)$ for any vector field $X$.

The first observation on $T$ for  a  generalized $(\lambda, n+m)$-Einstein manifold $(M^n, g, f)$ is the following.

\begin{lem}\label{lem202-3-3-4}
Let $(M^n, g, f)$ be a  generalized $(\lambda, n+m)$-Einstein manifold.
Then 
\bea
fC = m {\tilde i}_{\n f} {\mathcal W} - (m+n-2)T.\label{eqn7}
\eea
Here ${\tilde i}_{\n f} {\mathcal W}$ is defined by ${\tilde i}_{\n f} {\mathcal W}(X, Y, Z) = {\mathcal W}(X, Y, Z, \n f)$.
\end{lem}
\begin{proof}
Taking $d^D$ in (\ref{eqn5}), we obtain
$$
df\wedge z + f d^D z = m {\tilde i}_{\n f}{\mathcal R} - \frac{m}{n} d\Delta f \wedge g,
$$
where ${\mathcal R}$ denotes the Riemannian curvature tensor.
From the following curvature decomposition
\be
{\mathcal R} =  {\mathcal W} + \frac{s}{2n(n-1)} g\owedge g + \frac{1}{n-2} z\owedge g,\label{eqn2025-2-19-5}
\ee
we have
$$
{\tilde i}_{\n f} {\mathcal R}= {\tilde i}_{\n f}{\mathcal W}  - \frac{s}{n(n-1)}df \wedge g - \frac{1}{n-2}df\wedge z - \frac{1}{n-2}i_{\n f}z \wedge g.
$$
By (\ref{eqn2020-11-12-2}) and above,  we have
\be
fC &=& f d^Dz + \frac{n-2}{2n(n-1)} fds \wedge g \nonumber\\
&=&
m {\tilde i}_{\n f}{\mathcal R} - \frac{m}{n} d\Delta f \wedge g -df\wedge z + \frac{n-2}{2n(n-1)} fds \wedge g\nonumber\\
&=&
m {\tilde i}_{\n f}{\mathcal W} -\frac{m+n-2}{n-2}df \wedge z - \frac{m}{n-2} i_{\n f}z \wedge g\nonumber\\
&&\quad 
 + \frac{1}{n-1} \left(\frac{n-2}{2n} fds \wedge g - \frac{m}{n} s df \wedge g\right)
  - \frac{m}{n} d\Delta f \wedge g.\label{eqn2021-7-7-1}
\ee
Next, by taking the divergence operator $\d$ in (\ref{eqn5}), we have
\bea
-i_{\n f}z + f \d z = m(-i_{\n f} r - d\Delta f ) + \frac{m}{n} d\Delta f. \label{eqn2021-4-19-1}
\eea
Since $\d z = - \frac{n-2}{2n} ds$ and $i_{\n f}r = i_{\n f} z + \frac{s}{n} df$, we obtain
$$
(m-1)i_{\n f}z = \frac{n-2}{2n} f ds\wedge g - \frac{m}{n} s df \wedge g 
- \frac{m(n-1)}{n}d\Delta f\wedge g.
$$
Substituting this into (\ref{eqn2021-7-7-1}), we obtain
\bea
fC &=&  m {\tilde i}_{\n f}{\mathcal W} -\frac{m+n-2}{n-2}df \wedge z - \frac{m+n-2}{(n-1)(n-2)} i_{\n f}z \wedge g  \nonumber\\
&=&
m {\tilde i}_{\n f}{\mathcal W} -(m+n-2)T.\label{eqn8}
\eea
\end{proof}

For the Weyl tensor ${\mathcal W}$ and the trace-less Ricci tensor $z$, the symmetric $2$-tensor
$\mathring{\mathcal W}z$ is defined by
$$
\mathring{\mathcal W}z(X, Y) = z({\mathcal W}(X, e_i)Y, e_i)
$$
for a local frame $\{e_i\}$.

\begin{lem}\label{lem2018-2-10-1}
Let $(M^n, g, f)$ be a  generalized $(\lambda, n+m)$-Einstein manifold. Then
$$
\d ({\tilde i}_{\n f} {\mathcal W})  = - \frac{n-3}{n-2} \widehat C +  
\frac{f}{m} {\mathring {\mathcal W}}z,
$$
where ${\widetilde C}$ is a $2$-tensor defined as
$$
{\widehat C}(X, Y) = C(Y, \n f, X)
$$
for any vectors $X, Y$.
\end{lem}
\begin{proof}
Choose a local orthonormal frame $\{e_i\}$ which is normal at a point. Then, at the point,
it follows from definition and  (\ref{eqn5}), (\ref{eqn2016-12-3-16})  that
\bea
 \d ({\tilde i}_{\n f} {\mathcal W}) (X, Y)  &=&
 -D_{e_i}{\mathcal W}(e_i, X, Y, \n f) - {\mathcal W}(e_i, X, Y, D_{e_i}\n f)\\
 &=&
 \d {\mathcal W}(X, Y, \n f) - Ddf(e_i, e_j)  {\mathcal W}(e_i, X, Y, e_j)\\
 &=& 
 \d {\mathcal W}(X, Y, \n f) - \frac{f}{m} z(e_i, e_j) {\mathcal W}(e_i, X, Y, e_j)\qquad\quad (\because {\rm tr}_{14} {\mathcal W} =0)\\
 &=&
 - \frac{n-3}{n-2} C (Y, \n f, X) - \frac{f}{m} z(e_i, e_j) {\mathcal W}(e_i, X, Y, e_j) \\
 &=&
      - \frac{n-3}{n-2}C(Y, \n f, X) + \frac{f}{m} {\mathring {\mathcal W}}z (X, Y).
 \eea 

\end{proof}

For  a closed generalized $(\lambda, n+m)$-Einstein  manifold $(M^n, g, f)$, 
the vanishing of the tensor $T$ shows that $(M, g)$ satisfies  some constrained geometric structures.
 We say that $(M, g)$ has harmonic Weyl  curvature if $\d{\mathcal W} = 0$, and is Bach-flat 
 if the  Bach tensor $B$ defined by 
 $B = \frac{1}{n-3}\d^D\d{\mathcal W} + \frac{1}{n-2}{\mathring {\mathcal W}}z$ 
 vanishes. It is easy to see, from definition and 
(\ref{eqn2016-12-3-16}), that $(n-2)B = - \d C + \mathring{\mathcal W}z$.

\begin{thm} \label{lem23}
Let $(M^n, g, f)$ be a closed generalized $(\lambda, n+m)$-Einstein manifold.
 If $T=0$, then $(M, g)$ has harmonic Weyl curvature and is Bach-flat. 
\end{thm}
\begin{proof} 
By the definition of $T$, for an orthonormal frame $\{e_i\}_{1\leq i\leq n}$ 
with {$e_1=N = \frac{\n f}{|\n f|}$}, we have
\bea
(n-2)i_{\nabla f}T(e_j, e_k)&=& |\nabla f|^2z(e_j, e_k)+\frac 1{n-1}z(\nabla f, \nabla f)\d_{jk}\\
& & -\frac 1{n-1}z(\nabla f, e_j)df(e_k)-z(\nabla f, e_k)df(e_j).
\eea
Since $T=0$ by assumption, for {$2\leq j, k\leq n$}
\be
z_{jk}=z(e_j, e_k)= -\frac 1{n-1}z(N, N) \, \delta_{jk}. \label{eq1118}
\ee 
Also, for {$2\leq j\leq n$} 
$$
0=(n-2)i_{\nabla f}T(e_j, N)=  \frac {n-2}{n-1}\, z(e_j, N)|\nabla f|^2, 
$$
implying that
\be
 z(e_j, N)=0 \label{eq1119}
 \ee
for {$2\leq j\leq n$}. Letting $\a:= z(N, N)$, we have, by (\ref{eq1118}) and (\ref{eq1119}), 
\be
\langle i_{\nabla f}C, z\rangle = -\frac {\alpha}{n-1}\sum_{j=1}^{n-1} C(\nabla f, e_j, e_j) =\frac {\alpha}{n-1}C(\nabla f, N, N)=0.
\label{eq1120}
\ee
On the other hand, since 
\be 
f\, C = m\tilde{i}_{\nabla f}{ {\mathcal W}} \label{eq1109}
\ee
by  Lemma~\ref{lem202-3-3-4},   we have   $C(X, Y, \n f) = 0$ for any vector fields $X$ and $Y$. This  shows
$$ 
C(Y, \nabla f, X)+C(\nabla f, X, Y) = 0
$$
since the cyclic summation of $C$ is vanishing, which implies
\be
\widehat C = -i_{\n f}C.\label{eqn11}
\ee
By taking the divergence $\d$ of (\ref{eq1109}) and using Lemma~\ref{lem2018-2-10-1} and (\ref{eqn11}), we have 
\be
f\delta C = \frac{n-2+m(n-3)}{n-2} i_{\n f}C  + f\mathring{ {\mathcal W}}z. \label{eq1121}
\ee
It follows from the definition of $\mathring{ {\mathcal W}}z$ together with (\ref{eq1109}) that
$$
\mathring{ {\mathcal W}}z (\n f, X) = - \frac{f}{m} \langle i_XC, z\rangle 
$$
for any vector field $X$. In particular, by (\ref{eq1120}), we have
\be
\mathring{ {\mathcal W}}z(\nabla f, \nabla f) = -\frac{f}{m}\langle i_{\nabla f}C , z\rangle  =0\label{eqn12}.
\ee
Consequently, by (\ref{eq1121}) and  (\ref{eqn12})
\bea 
\delta C(N,N)= \mathring{ {\mathcal W}}z(N,N) =0.
\eea
Therefore, by (\ref{eq1120}), (\ref{eq1121}) and (\ref{eqn12}) again,
\bea
f\langle \delta C, z\rangle 
&=&
 \frac {n-2+m(n-3)}{n-2} f\langle i_{\nabla f}C, z\rangle +f\langle \mathring{ {\mathcal W}}z, z\rangle
= f\sum_{{2\leq j,k\leq n}}\mathring{ {\mathcal W}}z(e_j, e_k)z_{jk}\\
&=&-\frac {\a f}{n-1} \sum_{{2\leq j\leq n}}\mathring{ {\mathcal W}}z(e_j, e_j)
=\frac {\alpha f}{n-1}\mathring{ {\mathcal W}}z(N,N)=0,
\eea
implying that  
$$
\langle \delta C, z\rangle =0.
$$ 
Hence, from
$$ 
0=\int_M \langle \delta C, z\rangle= \frac 12 \int_M |C|^2,
$$
we have $C=0$, and hence ${\tilde i}_{\n f } {\mathcal W} = 0  ={ \mathring {\mathcal {\mathcal W}}}z$ by (\ref{eq1109}) and (\ref{eq1121}). 
Therefore
$$
\d {\mathcal W} = - \frac{n-3}{n-2} C = 0\quad \mbox{and}\quad (n-2)B = - \d C +{\mathring {\mathcal W}}z = 0.
$$
\end{proof}


Now, for  a closed generalized $(\lambda, n+m)$-Einstein  manifold $(M^n, g, f)$,  we define a $2$-form $\omega$ by
\be
\omega := df \wedge i_{\nabla f}z.\label{eqn20}
\ee
By the definition of $T$ and Lemma~\ref{lem202-3-3-4}, 
\bea
\quad T(X,Y, \nabla f)=\frac 1{n-1}df\wedge i_{\nabla f}z(X,Y)=\frac 1{n-1}\omega (X,Y)
 =-\frac {f}{m+n-2}\, \tilde{i}_{\nabla f}C(X,Y)\label{eq1001}
\eea
for any vector fields $X$ and $Y$.  Thus, we have
\be 
\omega = (n-1)\tilde{i}_{\nabla f}T =- \frac{n-1}{m+n-2} f \,\tilde{i}_{\nabla f}C. \label{eqn2021-3-17-1}
\ee
Let $\{e_i\}_{i=1}^n$ be a local orthonormal frame with {$e_1=N=\nabla f/|\nabla f|$}. 
It is clear that 
$$
\omega (e_j, e_k)=0
$$ 
for {$2 \leq j,k\leq n$} by  definition of $\o$. Thus, if $\tilde{i}_{\nabla f}C(N, e_i)=0$ for
 {$2\leq i\leq n$}, then $\omega =0$.

\begin{lem}\label{lem31}
As a $2$-form, we have
$$ 
\tilde{i}_{\nabla f}C=di_{\nabla f}z - \frac{n-2}{2n(n-1)} df \wedge ds.
$$
In particular, if $(M, g)$ has constant scalar curvature, then $\tilde{i}_{\nabla f}C$ is an exact form. 
\end{lem}
\begin{proof} 
Choose a local orthonormal frame $\{e_i\}$ which is normal at a point $p \in M$, 
and let $\{\theta ^i\}$ be its dual coframe so that $d\theta^i\vert_{p}=0$.  
Since $i_{\nabla f}z=\sum_{l,k=1}^n f_l z_{lk}\theta ^k$ with $e_l(f) = f_l$, $e_l(s) = s_l$ and
$z(e_l, e_k) = z_{lk}$, by  (\ref{eqn2020-11-12-2}) and symmetric property of the Hessian of $f$ and $z$, we have
\bea
di_{\n f}z &=& 
\sum_{j,k} \sum_{l} (f_{lj} z_{lk} + f_l z_{lk;j}) \theta^j \wedge \theta^k \\
&=&
\sum_{j<k} \sum_{l}  f_l (z_{lk;j} - z_{lj;k}) \theta^j \wedge \theta^k \\
&=&
 \sum_{j<k} \sum_{l} f_l \left[C_{jkl} -\frac{n-2}{2n(n-1)} \left(s_{j}\d_{kl}- s_{k}\d_{jl}\right)\right]   \theta^j \wedge \theta^k\\
&=&
 {\tilde i}_{\n f}C- \frac{n-2}{2n(n-1)}  \sum_{j<k} (f_ks_j-f_js_k) \theta^j\wedge \theta^k.
 \eea
\end{proof}

\begin{lem}\label{lem32}
$\omega$ is a closed $2$-form, i.e., $d\omega =0$.
\end{lem}
\begin{proof}
Choose a local orthonormal  frame $\{e_i\}$ with $e_1 = N = {\n f}/{|\n f|}$, and let $\{\theta^i\}$ be its dual coframe. 
Since $df  = |\n f|\theta^1$, by Lemma~\ref{lem31} 
\bea
di_{\n f}z 
&=& 
\sum_{j<k} \sum_{l} f_l C_{jkl}  \theta^j \wedge \theta^k 
 - \frac{n-2}{2n(n-1)}  \sum_{j<k} \sum_{l} (f_ks_j-f_js_k) \theta^j\wedge \theta^k\\
&=&
 |\n f| \sum_{k=2}^n   C_{1k1}  \theta^1 \wedge \theta^k
+ \frac{n-2}{2n(n-1)}  |\n f|\sum_{k=2}^n s_k
\theta^1\wedge \theta^k. 
\eea
Thus, by taking the exterior derivative of $\omega$ in (\ref{eqn20}), we have
$$
d\o = - df \wedge d i_{\n f}z = 0.
$$
\end{proof}

\section{closed generalized $(\lambda, n+m)$-Einstein  manifolds with $\omega=0$}

In this section, we assume that $(M^n, g, f)$ is a  closed generalized $(\lambda, n+m)$-Einstein 
manifold with $f^{-1}(0) \ne \emptyset$  satisfying
$\omega =df\wedge i_{\n f}z =0$. Then, it is easy to see that 
\bea
z(\nabla f, X)=0\label{eqn2019-5-27-2}
\eea
for any vector $X$ orthogonal to $\nabla f$ by  plugging $(\nabla f, X)$ into $\omega$. Thus, we may write 
$$
i_{\nabla f}z =\a df,\quad \mbox{where} \,\, \a = z(N, N), \,\, 
N = \frac{\n f}{|\n f|}
$$
as a $1$-form. Note that the function $\a$ is well-defined only on the set $M\setminus {\rm Crit}(f)$, but from $|\a| \le |z|$, it can be extended on the whole $M$ as a $C^0$-function.
Here ${\rm Crit}(f)$ denotes the set of all critical points of $f$.

In this section we will prove that there are no critical points of $f$ except at minimum
and maximum points of  $f$ in $M$, and either the {minimum} set $f^{-1}(a)$ for $a = \min_M f$
 and the maximum set $f^{-1}(b)$ for $b=\max_{M} f$ are both totally geodesic, or each set consists  of only a single point. 

First of all, recall that  $|\nabla f |$ is constant on each level set of $f$, since
$$ 
\frac 12 X(|\nabla f|^2)= \frac{f}{m}\, z(X, \nabla f) -\frac 1{n} (\Delta f) \langle X, \nabla f\rangle =0 
$$
for any tangent vector $X$ to a level set $f^{-1}(t)$. Moreover, on the set $f^{-1}(0)$, we have
$$ 
\nabla f(|\nabla f|^2)= \frac{f}{m}\, z(\n f, \nabla f) -\frac f{mn} (s- n\lambda) =0.
$$

\begin{lem}\label{lem2021-4-20-5}
Let $(M^n, g, f)$ be a  generalized $(\lambda, n+m)$-Einstein manifold satisfying  $\o = 0$. 
Then  $D_NN=0$ and the function $\a$ is constant along each level hypersurface $f^{-1}(t)$ of $f$.
\end{lem}
\begin{proof}
Note that $\dis{N(|\n f|) = Ddf(N, N) = \frac{\a f}{m} +  \frac{\Delta f}{n}}$ and
$$
N\left(\frac{1}{|\n f|}\right) = - \frac{1}{|\n f|^2} \left(\frac{\a f}{m}+ \frac{\Delta f}{n}\right).
$$
Since $i_N z = \a \n f$ as a vector field,  we have
$$
D_NN = N\left(\frac{1}{|\n f|}\right) \n f + \frac{1}{|\n f|} D_N df =0.
$$
Now, let $X$ be a vector field orthogonal to $\n f$. 
Since $D_NN=0$, we have
$g(D_NX, N) = - g(X, D_NN) = 0$ and so
\bea
D_N z(X, \n f) = - z(D_N X, \n f)-z(X, D_N \n f) = 0.\label{eqn2020-8-26-1}
\eea
Since $\tilde{i}_{\nabla f}C=0$ by (\ref{eqn2021-3-17-1}), we have
$$
0=C(X, N, \nabla f)=D_X z(N, \nabla f)-D_{N}z(X, \nabla f) = D_X z(N, \nabla f)=|\nabla f|X(\a),
$$
implying that $\a $ is a constant on $f^{-1}(t)$.

\end{proof}

\begin{remark}
The property $D_NN=0$ also implies that $ [X, N]$ is orthogonal to $\n f$. Using this, one can
 show that $|z|^2$ is also constant along each level hypersurface $f^{-1}(t)$ of $f$.
\end{remark}

\begin{lem}\label{lem2021-5-15-6}
Let $(M^n, g, f)$ be a  closed generalized $(\lambda, n+m)$-Einstein 
manifold with $f^{-1}(0) \ne \emptyset$.
 Assume that $\o = 0$. Then $\a$ is constant and, in particular,
$\langle \n s, \n f\rangle = 0$ on the set  $f^{-1}(0)$.
\end{lem}

\begin{proof}
To show $\a$ is constant, it suffices to prove $\langle \n\alpha, \n f\rangle =0$ on $M$  from
Lemma~\ref{lem2021-4-20-5}. 
Since ${\tilde i}_{\n f}C = 0$ by (\ref{eqn2021-3-17-1}) and  $di_{\n f}z = d\a\wedge df$,  
we can see,  in the proof of Lemma~\ref{lem31} and Lemma~\ref{lem32}, that
$$
d\a = \frac{n-2}{2n(n-1)}  \sum_{k=2}^n s_{k} \theta^k,
$$
which implies that $\langle \n \a, \n f\rangle = 0$. (Note that by taking $N = e_1$, we have
$f_1= |\n f|, f_k = 0$ for $k \ge 2$.)

Now since $\a$ is constant and $i_{\n f}z = \a df$, we have
$\d i_{\n f}z =  - \a\Delta f$.
Also, from the definition of divergence, we have
\bea
\d i_{\n f}z  = - \frac{n-2}{2n} \langle \n s, \n f\rangle - \frac{f}{m}|z|^2,\label{eqn2021-5-31-1}
\eea
and hence
\be
\frac{n-2}{2n} \langle \n s, \n f\rangle = \a\Delta f - \frac{f}{m}|z|^2.\label{eqn2021-5-15-11}
\ee
In particular, we have $\langle \n s, \n f\rangle =0$ on the set $f^{-1}(0)$ since
$\Delta f = 0$ on $f^{-1}(0)$.
\end{proof}

\begin{lem}\label{cor2021-5-15-7}
Let $(M^n, g, f)$ be a  closed generalized $(\lambda, n+m)$-Einstein 
manifold with $f^{-1}(0) \ne \emptyset$.
 Assume that $\o = 0$. If $\langle \n s, \n f\rangle \ge 0$  on $M$, then $\a$ is nonpositive constant. 
\end{lem}

\begin{proof}
First, $\a$ is constant by Lemma~\ref{lem2021-5-15-6}.
From (\ref{eqn2021-5-15-11}), we have the following  inequality
\bea
\frac{n-2}{2n}\int_{f>0} \langle \n s, \n f\rangle = - \a \int_{f=0}|\n f| 
- \frac{1}{m}\int_{f>0} f|z|^2,\label{eqn2021-5-15-10}
\eea
which shows that $\a$ is nonpositive on the set $f^{-1}(0)$. Since $\a$ is constant, $\a \le 0$ on the whole $M$. 
\end{proof}

\begin{lem}
Let $(M^n, g, f)$ be a closed generalized $(\lambda, n+m)$-Einstein manifold with $f^{-1}(0) \ne \emptyset$.
 Assume that $\o = 0$. 
Then $\n s$ is parallel to $\n f$. In particular $\n s = 0$ on the set $f^{-1}(0)$.
\end{lem}
\begin{proof}
Since $i_{\n f}z = \a df$ and $\a$ is constant,  by taking the differential operator $d$ of
 (\ref{eqn2021-4-19-15}), we obtain
$$
\left(\frac{m-1}{n} + \frac12 \right) ds \wedge df =0,
$$
which shows $\n s$ is parallel to $\n f$. In particular, since $\langle \n s, \n f\rangle =0$
on {$f^{-1}(0)$ by  Lemma~\ref{lem2021-5-15-6}}, we have
$\n s = 0$ on the set $f^{-1}(0)$ because $\n f \ne 0$ on $f^{-1}(0)$ by  Lemma~\ref{prop2021-5-15-2}.

\end{proof}

\begin{lem}\label{lem2021-5-4-1} 
Let $(M^n, g, f)$ be a  closed generalized $(\lambda, n+m)$-Einstein  manifold.
 Assume that  $\o = 0$ and $s$ is  a (positive) constant. 
Then we have
\be
\lambda = \frac{m+n-1}{n(n-1)}s + \frac{m-1}{n-1}\a. \label{eqn2021-5-4-2}
\ee
In particular, $\lambda$ is a positive constant and  $s < n \lambda$ on $M$.
\end{lem} 
\begin{proof}
Since $i_{\n f}z = \a df$ and $\a$ is constant, it follows from Lemma~\ref{lem2021-4-19-100} that the function
$$
(m-1)\a f+ \frac{m-1}{n}sf - (n-1)\mu 
$$
is constant on $M$. Considering the set $f^{-1}(0)$, it must be vanishing. Finally, substituting $\mu = \lambda f - \frac{s}{n-1}f$, 
we obtain
$$
\lambda = \frac{m+n-1}{n(n-1)}s + \frac{m-1}{n-1}\a.
$$
In particular, $\lambda$ is constant. If $\lambda \le 0$, 
from $\Delta f = \frac{f}{m}(s-n \lambda)$, the function $f$ is subharmonic 
on the set $f \ge 0$, which is a contradiction to the maximum principle.
\end{proof}

\begin{cor}\label{cor2025-2-19-1}
Let $(M^n, g, f)$ be a closed generalized $(\lambda, n+m)$-Einstein manifold of constant scalar curvature
with $f^{-1}(0) \ne \emptyset.$
If $z(\n f, \n f) = 0$ on $M$, then $(M, g)$ is isometric to a sphere ${\Bbb S}^n$.

\end{cor}
\begin{proof}
If $\a = 0$, then  we have
$z = 0$ from (\ref{eqn2021-5-15-11}) on the set $f >0$. By elliptic theory, we have $z = 0$ on the whole $M$ which means $(M, g)$ is Einstein.
The equation (\ref{eqn1}) becomes
\be
mDdf = \left(\frac{s}{n}-\lambda\right) fg.\label{eqn2025-2-19-2}
\ee
By Obata's result in \cite{oba}, $M$ is isometric to a sphere ${\Bbb S}^n$ since $\frac{s}{n}-\lambda <0$ by Lemma~\ref{lem2021-5-4-1}.
\end{proof}

\begin{lem}\label{lem2021-5-31-2}
 Let $(M^n, g, f)$ be a  closed generalized $(\lambda, n+m)$-Einstein  manifold with 
 $f^{-1}(0) \ne \emptyset$. If the scalar curvature $s$ is constant and  $\o=0$,  we have
 $$
 {\rm Ric}(N,  N) \ge 0
 $$
on the set $M$.
 \end{lem}
 \begin{proof}
 Suppose, on the contrary,  that ${\rm Ric}(N, N) < 0$ at a point $x\in M$. Since ${\rm Ric}(N, N) = \a + \frac{s}{n}$ is constant
 by Lemma~\ref{lem2021-5-15-6},  ${\rm Ric}(N, N) < 0$ on the whole $M$. Considering a connected component
 $\Gamma$ of $f^{-1}(0)$ which is totally geodesic by Lemma~\ref{lem202-5-15-3}, we have
 $$
 \int_\Gamma \left[|\n^\Gamma \vp|^2 - {\rm Ric}(N, N)\vp^2\right] \ge 0
 $$
 for any smooth function $\vp$ defined on $\Gamma$.
 By Fredholm alternative (cf. \cite{f-s}, Theorem 1), 
there exists a  positive function $\vp>0$ on $\Gamma$ satisfying 
$$
\Delta^\Gamma \vp + {\rm Ric}(N, N)\vp =0.
$$
However,  it follows from the maximum principle that $\vp$ must be a constant which is impossible.
 \end{proof}
 
 \begin{lem}\label{eqn2021-5-31-1}
  Let $(M^n, g, f)$ be a  closed generalized $(\lambda, n+m)$-Einstein  manifold with 
  $f^{-1}(0) \ne \emptyset$. If the  scalar curvature $s$ is constant and  $\o = 0$, then
  $$
  {\rm Ric}(N,  N) = \frac{(n-1)\lambda - s}{m-1}.
  $$
  In particular, we have $(n-1)\lambda  \ge s$.
 \end{lem}
 \begin{proof}
 Following \cite{h-p-w}, we define a tensor $P:= {\rm Ric} - \rho g$ with $\rho =\frac{(n-1)\lambda-s}{m-1}$. By taking the divergence of (\ref{eqn1}), we can obtain (cf. \cite{f-san})
$$
\frac{f\n s}{2(m-1)} - \frac{(n-1)f\n \lambda}{m-1} + P(\n f, \cdot) = 0.
$$
Since   $\lambda$ is also constant by  Lemma~\ref{lem2021-5-4-1}, we have
$$
P(\n f, \cdot) = 0.
$$
In particular, 
\bea
{\rm Ric}(N, N) = \rho  =\frac{(n-1)\lambda-s}{m-1}.\label{eqn2021-5-27-2-1}
\eea
 The  inequality $(n-1)\lambda  \ge s$ follows from Lemma~\ref{lem2021-5-31-2}.
 \end{proof}

\begin{lem}\label{lem2021-5-15-16} 
Let $(M^n, g, f)$ be a  closed generalized $(\lambda, n+m)$-Einstein 
manifold with $f^{-1}(0)\ne \emptyset$.
 Assume that  $\o = 0$ and $s$ is  a (positive) constant. 
Then there are no critical points of $f$ except at the minimum and maximum points of $f$.
\end{lem} 
\begin{proof}
We have $s>0$, $\a \le 0$, $s-n\lambda <0$ and these are all constants.
If $\a = 0$, then  $M$ is Einstein by (\ref{eqn2021-5-15-11}) and so is isometric to a sphere by the proof of Corollary~\ref{cor2021-5-15-7}.  So, we may assume   $\a <0$.  Recall that, by (\ref{eqn2}) and (\ref{eqn2021-5-15-11})
\be
\Delta f = \frac{s-n\lambda}{m} f\quad \mbox{and}\quad \a \Delta f = \frac{f}{m}|z|^2\label{eqn2021-4-20-8}
\ee
and so $|z|^2 = \a(s-n \lambda) $ is also a positive constant. 
It follows from the Bochner-Weitzenb\"ock formula that
$$
\frac{1}{2}\Delta |\n f|^2 = |Ddf|^2 + \langle \n \Delta f, \n f\rangle + z(\n f, \n f) + \frac{s}{n}|\n f|^2.
$$
From (\ref{eqn5}), we have
$$
z(\n f, \n f) = \frac{m}{f} Ddf(\n f, \n f) - \frac{m}{n} \frac{\Delta f}{f}|\n f|^2 = \frac{m}{2f}\n f(|\n f|^2) - \frac{s-n\lambda}{n}|\n f|^2.
$$
Thus,
\be
\frac{1}{2}\Delta |\n f|^2 -\frac{m}{2f}\n f(|\n f|^2)  = |Ddf|^2 
+ \left(\frac{s-n\lambda}{m} + \lambda\right)|\n f|^2.\label{eqn2021-4-20-10}
\ee
Note that (\ref{eqn2021-4-20-10}) is valid only on the set  $f>0$  or $f<0$.

\vspace{.12in}
\noindent
{\bf Assertion:} We claim $\dis{\frac{s-n\lambda}{m} + \lambda > 0.}$

\vspace{.1in}
\noindent
From Lemma~\ref{lem2021-5-4-1}, we have
$$
\frac{s-n\lambda}{m} = - \frac{s}{n-1} - \frac{n(m-1)}{m(n-1)}\a.
$$
So, by Lemma~\ref{eqn2021-5-31-1},
$$
\frac{s-n\lambda}{m} +\lambda = \lambda - \frac{s}{n-1} - \frac{n(m-1)}{m(n-1)}\a >0.
$$

\vspace{.1in}
\noindent
Now applying the  maximum principle to (\ref{eqn2021-4-20-10}) on the set $f>0$,  
the function $|\nabla f|^2$ cannot have its local maximum on the set $f>0$. 

Now, suppose that there is a critical point $p$ of $f$ with $0< f(p)=c<b=\max f$. 
Since $|\nabla f|$ is constant on each level set of $f$, $|\nabla f|=0$ on $f^{-1}(c)$. 
However, this implies that there should be a local maximum of $|\nabla f|^2$ in the set 
$\{x\in M\, \vert\, c <f(x) <b\}$, which is impossible by the above maximum principle.
Since $|\n f|^2$ cannot have its local maximum in the set $f<0$, the exactly same argument as above shows that $f$ cannot have its critical points  in the set $f<0$.
\end{proof}

Let $(M, g, f)$ be a  closed generalized $(\lambda, n+m)$-Einstein 
manifold with $f^{-1}(0) \ne \emptyset$.
 Assume that  $\o = 0$ and $s$ is  a (positive) constant.  Let ${\min_M f = a}$ and
${\max_M f = b}$ with $a<0<b$.
Lemma~\ref{lem2021-5-15-16}   shows that the sets  $f^{-1}(a)$ and $f^{-1}(b)$ are both connected, 
and either each set  is a single point or a hypersurface by the isotopy lemma. In fact, if $f^{-1}(a)$ has at least two components,  $f$ may have
 a critical point other than $f^{-1}(a)$. 
 Lemma~\ref{lem2021-5-15-16}   also shows that if $f^{-1}(0)$ is disconnected, 
then it has only two connected components.
Furthermore, if $f^{-1}(0)$ is connected, then both
 $f^{-1}(a)$ and $f^{-1}(b)$ consist of only a single point, and in this case each level hypersurface including $f^{-1}(0)$ is homotopically  an $(n-1)$-sphere ${\Bbb S}^{n-1}$.
 

The following result shows that  if  $f^{-1}(a)$ contains only a single point, then so does $f^{-1}(b)$, and vice versa. 
Moreover, in case of hypersurface, it should be  a totally geodesic stable minimal hypersurface.

\begin{lem}\label{lem2021-5-15-17}
Let $(M^n, g, f)$ be a  closed generalized $(\lambda, n+m)$-Einstein 
manifold with $f^{-1}(0)\ne \emptyset$.
 Assume that  $\o = 0$ and $s$ is  a (positive) constant.  Let ${\min_M f = a}$ and
${\max_M f = b}$ with $a<0<b$.
If  $f^{-1}(a)$  contains only a single point, then so does  $f^{-1}(b)$, and vice versa.
Furthermore, if $f^{-1}(a)$ is a single point, then 
 every level set $f^{-1}(t)$ except $t=a$ and $t=b$ is
a hypersurface and is homotopically a sphere ${\Bbb S}^{n-1}$.
\end{lem}
\begin{proof}
First of all, if $\a=0$, then $(M, g)$ is Einstein by Corollary~\ref{cor2025-2-19-1}. Moreover, since $\Delta f = - \frac{s}{n-1}f$
by Lemma~\ref{lem2021-5-4-1}  and (\ref{eqn2025-2-19-2}), $f$ is a height function and so satisfies our property. 
Thus, we may assume that $\a<0$. Suppose that $f^{-1}(a)$ is a hypersurface, but $f^{-1}(b)$ is not a hypersurface. It follows from
Lemma~\ref{lem2021-5-15-16} together with isotopy lemma that the set $f^{-1}(b)$ consists of
only two points, and the set $M-f^{-1}(a)$ has two connected components. Now, by (\ref{eqn2021-5-15-11}), 
\be
\Delta f = \frac{f}{m\a}|z|^2.\label{eqn202105-15-13}
\ee
Applying the maximum {principle} to  (\ref{eqn202105-15-13}) on each connected component of
$M - f^{-1}(a)$, $f$ must attain  its maximum on the boundary $f^{-1}(a)$, which is a contradiction.
\end{proof}

\begin{lem}\label{lem2021-3-3-2} 
Let $(M^n, g, f)$ be a  closed generalized $(\lambda, n+m)$-Einstein 
manifold with $f^{-1}(0)\ne \emptyset$.
 Assume that   $\o = 0$ and $s$ is  a (positive) constant. 
Suppose that $\Sigma:= f^{-1}(a)$ $[$$\Sigma = f^{-1}(b)$$]$  is a hypersurface for 
$a = \min_M f$ $[$$b = \max_M f$$]$  and assume $\nu$ is a unit normal vector field on $\Sigma$. Then we have the following.
\begin{itemize}
\item[(1)] $\a  = z(\nu, \nu)= -\frac{n-1}{n} \left(n\lambda - s\right)  < 0$ and
$Ddf_p(X, X)=0$ for a vector $X$ orthogonal to $\nu$ at any point $p \in \Sigma$.
\item[(2)] $\Sigma$ is totally geodesic.
\end{itemize}
\end{lem} 
\begin{proof}
(1)  For a sufficiently small $\epsilon >0$, $f^{-1}(a-\epsilon)$ has two connected components 
$\Sigma_{\epsilon}^+$ and $\Sigma_{\epsilon}^-$. 
Let $\nu$ be a unit normal vector field on $\Sigma=f^{-1}(a)$. On a tubular neighborhood of 
$\Sigma$, $\nu$ can be extended smoothly to a vector field $\tilde{\nu}$ such that 
$\tilde{\nu}\vert_{\Sigma}=\nu$ with 
$\tilde{\nu}\vert_{\Sigma_{\epsilon}^+} =N=\frac {\nabla f}{|\nabla f|}$ 
and $\tilde{\nu}\vert_{\Sigma_{\epsilon}^-}=-N=-\frac {\nabla f}{|\nabla f|}$.

The Laplacian of $f$ on $f^{-1}(a-\epsilon)=\Sigma_{\epsilon}^-\cup \Sigma_{\epsilon}^+$ is given by
\be 
\tr f =\tr ' f+Ddf(\tilde \nu, \tilde \nu)+ H\langle \tilde \nu, \nabla f\rangle, \label{eq41}
\ee
where $\tr '$  and $H$ denote the intrinsic Laplacian and the mean curvature of $f^{-1}(a-\epsilon)$, respectively. Since $\nabla f=0$ at $\Sigma$ and $\tr 'f=0$ in (\ref{eq41}), 
by letting $\epsilon \to 0$, we have
\bea 
\tr f = Ddf(\nu,\nu) \label{eq42}
\eea
on $\Sigma$. It is clear that $H$ is bounded around  $\Sigma$. Thus, 
\bea
-\frac {n\lambda -s}{m} a= \Delta f = Ddf(\nu, \nu) = \frac{a}{m} z(\nu, \nu)
-\frac {n\lambda -s}{mn} a, \label{eqn2021-3-8-1}
\eea
implying that 
\be 
z(\nu, \nu)= -\frac{n-1}{n} \left(n\lambda - s\right)  < 0.  \label{eq43}
\ee
Now let $p \in \Sigma$ be a point and choose an orthonormal basis $\{e_i\}_{i=1}^n$ at
 $p$ with $e_1=\nu$. Since $\Sigma$ is the maximum set of $f$, we have
\be
Ddf_p(e_i, e_i) = \frac{a}{m} z_p(e_i, e_i)-\frac {n\lambda -s}{mn} a \ge  0\label{eqn2025-2-19-3}
\ee
for each $i,    2\leq i\leq n$, and so 
$$ z_p(e_i, e_i) \ge \frac {n\lambda -s}{n} = -\frac 1{n-1}z_p(\nu, \nu).
$$
Since $\sum_{i=2}^{n}z_p(e_i, e_i)=-z_p(\nu, \nu)$, this  together with (\ref{eq43}) implies that 
\bea 
z_p(e_i, e_i) = \frac {n\lambda -s}{n} > 0\label{eq44}
\eea
and hence, by (\ref{eqn2025-2-19-3}),
\bea 
Ddf_p(e_i, e_i)=0\label{eq45}
\eea 
on $\Sigma$ for each $i,  2\leq i\leq n$. In case of maximum set $\Sigma = f^{-1}(b)$ with
$b = \max_M f$, the inequalities are just reversed above and we have the same conclusion.

\vspace{.08in}
\noindent
(2) 
Since $z(\nu, X)=0$ for $X$ orthogonal to $\nu$ at $p\in \Sigma$, we may take the previously 
mentioned orthonormal basis $\{e_i\}_{i=1}^n$ so that $\{e_i\}_{i=2}^{n}$ are tangent to 
$\Sigma$.

For $e_2$, let $\gamma : [0,l)\to M$ be a unit speed geodesic 
such that $\gamma(0)=p$, $\gamma'(0)=e_2$ for some $l>0$. 
Defining  $\varphi (t)=f\circ \gamma(t)$, we have  $\varphi '(0)=df_p (\gamma'(0))=0$,  and by (1) above
$$
\varphi ''(0)=Ddf(\gamma'(0), \gamma'(0))= Ddf_p(e_2, e_2)=0.
$$
Note that
$$ 
\varphi''(t) =Ddf(\gamma'(t), \gamma'(t))=   \left[z(\gamma'(t), \gamma'(t))
+\frac {(s-n\lambda)}{n}\right]\frac{\varphi(t)}{m}.
$$
So, it follows from the uniqueness of ODE solution that $\vp$ vanishes identically, which implies
that $\gamma(t)$ stays in $\Sigma$. Since $e_2$ is an arbitrary tangent vector perpendicular to $\nu$, we can conclude that
$\Sigma$ is totally geodesic.

\end{proof}

Let $(M^n, g, f)$ be a  closed generalized $(\lambda, n+m)$-Einstein manifold 
with $f^{-1}(0) \ne \emptyset$.
 Assume that   $\o = 0$ and $s$ is  a (positive) constant. 
 By Lemma~\ref{eqn2021-5-31-1}, we have
\be
{\rm Ric}(N, N) = \rho  =\frac{(n-1)\lambda-s}{m-1} \ge 0 \label{eqn2021-5-27-2}
\ee
and so
\be
\a = z(N, N) = {\rm Ric}(N, N) -\frac{s}{n} = \frac{(n-1)\lambda-s}{m-1}-\frac{s}{n}.\label{eqn2021-5-27-1}
\ee
Furthermore, by (\ref{eq43}), we have
$$
 \frac{(n-1)\lambda-s}{m-1}-\frac{s}{n} = \a =- \frac{n-1}{n}(n \lambda -s)
$$
So, we obtain
\bea
(n-1)\lambda - s= 0\label{eqn2021-5-6-2}
\eea 
and hence
$$
{\rm Ric}(N, N) = 0.
$$
 By continuity, this equality hold on the whole $M$ including  
the set $\Sigma = f^{-1}(a)$ for $a = \min_M f$ or  $\Sigma = f^{-1}(b)$ for $b = \max_M f$
if we assume  $\Sigma$ is a hypersurface. In particular, by (\ref{eqn2021-5-27-1}), 
\bea
z(N, N) =  - \frac{s}{n}.
\eea
Since $\Sigma$ is totally geodesic, we have
$$
s_{\Sigma} = s - 2{\rm Ric}(\nu, \nu) = s .
$$

\begin{lem}\label{lem2021-5-6-3} 
Let $(M^n, g, f)$ be a  closed generalized $(\lambda, n+m)$-Einstein 
manifold with $f^{-1}(0)\ne \emptyset$.
 Assume that   $\o = 0$ and $s$ is  a (positive) constant.  
Suppose that $\Sigma:= f^{-1}(a)$ $[\Sigma = f^{-1}(b)]$ is a hypersurface for $a = \min_M f$
$[b = \max_M f]$ and assume $\nu$ is a unit normal vector field on $\Sigma$. Then we have 
$$
\lambda = \frac{s}{n-1},\quad \a = - \frac{s}{n}\quad \mbox{and}\quad
{\rm Ric}(N, N) = 0.
$$
Moreover, $\Sigma$ is stable.
\end{lem} 
\begin{proof}
The stability operator becomes
$$
\int_{\Sigma} \left\{|\n \varphi|^2 - (|A|^2+{\rm Ric}(\nu, \nu))\vp^2\right\}
=\int_{\Sigma} |\n \varphi|^2 \ge 0
$$
for any function $\varphi$ on $\Sigma$. Here $A$ denotes the second fundamental form of $\Sigma$.
\end{proof}

\begin{thm}\label{thm001} 
Let $(M^n, g, f)$ be a  closed generalized $(\lambda, n+m)$-Einstein 
manifold with $f^{-1}(0)\ne \emptyset$.
 Assume that   $\o = 0$ and $s$ is  a (positive) constant. If $f^{-1}(a)$ is a  hypersurface for
$a = \max_M f$ or $a = \min_M f$, then $(M, g)$ has harmonic Weyl curvature and is
Bach-flat.
\end{thm}
\begin{proof}
Note that
$$
|T|^2 = \frac{2} {(n-2)^2}{|\n f|^2}\left(|z|^2 - \frac{n}{n-1}\a^2\right).
$$
By (\ref{eqn2021-4-20-8}), we have
$$
|z|^2 = \a(s - n\lambda) = - \a \lambda.
$$
So,
\bea
|z|^2 - \frac{n}{n-1}\a^2 &=& - \frac{\a}{n-1}\left[(n-1) \lambda + n \a\right]\\
&=&
 - \frac{\a}{n-1}\left[(n-1) \lambda -s\right] =0,
\eea
which means $T=0$. The conclusion follows from Theorem~\ref{lem23}.
\end{proof}

Let $(M^n, g, f)$ be a  closed generalized $(\lambda, n+m)$-Einstein 
manifold with $f^{-1}(0) \ne \emptyset$.
 Assume that   $\o = 0$ and $s$ is  a (positive) constant. If $f^{-1}(a)$ is a hypersurface for
$a = \max_M f$ or $a = \min_M f$ so that $f^{-1}(0)$ has two connected components 
which are both totally geodesics, then $T=0$ and so we have
$$
z_{ij} = - \frac{\a}{n-1}\d_{ij} = \frac{s}{n(n-1)}\d_{ij} \quad\mbox{for $2 \le i, j \le n$}
$$
and
$$
z_{1k} = \a \d_{1k} = - \frac{s}{n}\d_{1k}\quad \mbox{for $1\le k \le n$}.
$$
That is,
\bea
{\rm Ric}_{ij} =  \frac{s}{n-1}\d_{ij} \quad\mbox{for $2 \le i, j \le n$}
\eea
and
\bea
{\rm Ric}_{1k} = 0.
\eea
In particular, on each level hypersurface $f^{-1}(t)$, we have
$$
\frac{s}{n(n-1)} g - z = 0
$$
and from $\frac{s}{n(n-1)} g(N, N) - z(N, N) = \frac{s}{n-1}$, we obtain
\be
\frac{s}{n(n-1)} g =  z + \frac{s}{n-1}\frac{df}{|df|}\otimes \frac{df}{|df|}.\label{eqn2021-5-1-1}
\ee
On the other hand, for the curvature tensor $\mathcal R$ with $N = \n f /|\n f|$, $R_N$ is defined  as follows 
$$
R_N(X, Y) = {\mathcal R}(X, N, Y, N)
$$
for any vector fields $X$ and $Y$.
From $C=0$ and Lemma~\ref{lem202-3-3-4}, we have
$$
i_{\n f}{\mathcal W} = 0.
$$
So, by the curvature decomposition (\ref{eqn2025-2-19-5}),  we can obtain $ R_N(X, Y) = 0$ for any vector fields $X$ and $Y$ orthogonal to $\n f$.
Since $\Sigma=f^{-1}(a)$ with $a = \max_M f$ or $a = \min_M f$ is totally geodesic by Lemma~\ref{lem2021-3-3-2}, it follows from the Gauss equation that
$$
{\rm Ric}^\Sigma(X, X) = \lambda - R(X, \nu, X, \nu) = \lambda,
$$
which shows that $\Sigma$ is an Einstein manifold with positive Ricci curvature. 
From the same reason, any connected component of
the set  $f^{-1}(0)$ is also Einstein by Lemma~\ref{lem202-3-3-4}.
In fact, since $(M, g)$ has harmonic Weyl curvature and $i_{\n f}{\mathcal W} = 0$, from a result in \cite{cat}
$g$ is locally a warped product of an interval with an Einstein manifold around any regular point of $f$.
In our case, from (\ref{eqn2021-5-1-1}), we can see that $g$ is, in fact, a product of a circle
${\Bbb S}^1$ with an Einstein manifold $\Sigma_0$ which is totally geodesic and Einstein 
with positive Ricci curvature.  Therefore, we have the following.

\begin{thm}
Let $m>1$ and let $(M^n, g, f)$ be a closed generalized $(\lambda, n+m)$-Einstein manifold of constant scalar curvature with $f^{-1}(0)\ne \emptyset$. If $\o=0$  and  $f^{-1}(0)$ is disconnected, then $f^{-1}(0)$ has only two connected components, and, up to finite cover and rescaling, $M$ is isometric to the product ${\Bbb S}^1 \times {\Sigma}^{n-1}$ of a circle and an $(n-1)$-dimensional Einstein manifold $\Sigma$ of positive Ricci curvature. 
\end{thm}

In particular the fundamental group of $\Sigma_0$ is finite 
by Myers' theorem. So if we assume $(M, g)$ has PIC, then $\Sigma_0$ has also PIC and 
up to finite cover, it is homeomorphic to a standard sphere ${\Bbb S}^{n-1}$ by a result in
 \cite{mm88}. Hence if the set $f^{-1}(0)$ is disconnected, then $M$ is isometric to 
${\Bbb S}^1\times {\Bbb S}^{n-1}$, up to finite cover and rescaling.


\begin{thm}\label{str001} 
Let $(M^n, g, f), n \ge 4$, be a  closed generalized $(\lambda, n+m)$-Einstein manifold of constant scalar curvature with $f^{-1}(0)\ne \emptyset$ and $\o=0$.  Assume that $(M, g)$ has PIC and    $f^{-1}(0)$ is disconnected. Then $f^{-1}(0)$ has only two connected components and  $M$ is isometric to ${\Bbb S}^1 \times {\Bbb S}^{n-1}$,  up to finite cover and rescaling.
 \end{thm} 
\begin{proof}
First, assume $n \ge 5$ so that $\dim(\Sigma) \ge 4$. As mentioned above, $\Sigma$ is  homeomorphic to a sphere, up to finite cover.
 Moreover, it is proved in \cite{BS} that $\Sigma$ has  constant positive sectional curvature since it has PIC and Einstein of positive Ricci curvature. Therefore, $\Sigma$ must be isometric to a sphere. Now, assume the dimension of $M$ is $4$. Since  $\Sigma$ is an Einstein $3$-manifold of positive Ricci curvature, it has a positive constant sectional curvature too.
\end{proof}

\section{compact generalized Einstein manifold  with connected zero set}

Throughout this section, we assume that  $(M^n, g, f)$ is a  closed generalized 
$(\lambda, n+m)$-Einstein  manifold with constant scalar curvature $s$. We also assume that
$(M, g)$ has  $\o =df \wedge i_{\n f} z =0$ and the zero set $f^{-1}(0)$ is connected. Under these hypotheses, we show that $M$ is isometric to a sphere ${\Bbb S}^n$,  up to finite cover and rescaling.

Recall that $\a \le 0$, and if $\a = 0$, or equivalently ${\rm Ric}(N, N) >0$, it follows from Lemma~\ref{lem2021-5-15-16}  and Lemma~\ref{lem2021-5-15-17} that
the minimum set $f^{-1}(a)$ with $a = \min_Mf$ and the maximum set
$f^{-1}(b)$ with $b = \max_M f$ consist of a single point, respectively, and
$f^{-1}(0)$ must be connected. In particular, we have ${\rm Ric}(N,  N) > 0$. 
The following shows that the converse is also true.

 \begin{lem}\label{lem2021-5-27-3}
 Let $(M^n, g, f)$ be a  closed generalized $(\lambda, n+m)$-Einstein  manifold with constant scalar curvature $s$. If $\o = 0$ and $f^{-1}(0)$ is connected, then we have
 $$
 {\rm Ric}(N,  N) > 0
 $$
on the set $M$.
 \end{lem}
 \begin{proof}
By Lemma~\ref{lem2021-5-31-2}, we have $ {\rm Ric}(N,  N) \ge 0$.
 Suppose that ${\rm Ric}(N, N) = 0$ on $M$ so that $\a = - \frac{s}{n}$ and 
 $(n-1)\lambda - s = 0$ by (\ref{eqn2021-5-27-1}). Then as in the proof of Theorem~\ref{thm001} and the argument just below it, we have $T=0$ in this case, and so the metric
 $g$ has the same form as (\ref{eqn2021-5-1-1}):
$$
\frac{s}{n(n-1)} g =  z + \frac{s}{n-1}\frac{df}{|df|}\otimes \frac{df}{|df|}
$$ 
which shows that $g$ is, in fact, a product metric. However, since $f^{-1}(0)$ is connected, 
 both minimum set $f^{-1}(a)$ and maximum set $f^{-1}(b)$ consist of a single point, respectively, the metric $g$ cannot be a product, a contradiction.
 \end{proof}

 Let $(M^n, g, f)$ be a  closed generalized $(\lambda, n+m)$-Einstein 
manifold.
 Assume the scalar curvature $s$ is   constant. If $(M, g)$ has PIC  and $f^{-1}(0)$ is connected,  then, from Lemma~\ref{lem2021-5-15-17} and \cite{mm88}, $M$ is homeomorphic to a sphere.
 
Now, we introduce a warped product metric involving
$\frac {df}{|\nabla f|}\otimes \frac {df}{|\nabla f|}$ as a fiber metric on each level
set $f^{-1}(c)$. Consider a warped product metric $\bar{g}$ on $M$ by
\be
\bar{g}= \frac {df}{|\nabla f|}\otimes \frac {df}{|\nabla f|}+|\nabla f|^2 g_{\Gamma},
\label{eqn2019-8-27-1}
\ee
where $g_{\Gamma}$ is the restriction of $g$ to  $\Gamma:= f^{-1}(0)$. 
Note that, from Lemma~\ref{lem2021-5-15-16} , 
the metric $\bar g$ is smooth on $M$ except, possibly at two points, the maximum and minimum points of $f$.

The following lemma shows that $\nabla f$ is a conformal  vector field with respect to the metric $\bar{g}$.

\begin{lem}\label{lemt1} 
 Let $(M, g, f)$ be a  closed generalized $(\lambda, n+m)$-Einstein 
manifold with constant scalar curvature $s$ and $\omega=0$. Then,
$$\frac 12 {\mathcal L}_{\nabla f}\bar{g} =N(|\nabla f|)\bar{g}= \frac 1n 
(\bar{\tr }f)\, \bar{g}.
$$
Here, $\mathcal L$ denotes the Lie derivative.
\end{lem}
\begin{proof}
Note that, by (\ref{eqn1}) we have,
$$ 
\frac 12 {\mathcal L}_{\nabla f} g  =D_gdf = \frac{f}{m}({\rm Ric}-\lambda g).
$$
 By the definition of Lie derivative, 
\bea
\frac 12{\mathcal L}_{\nabla f}(df \otimes df)(X,Y)&=& Ddf(X, \nabla f)df(Y)+df(X)Ddf(Y, \nabla f)\\
&=&   \frac{2f}{m}({\rm Ric}(N, N)-\lambda) \, df\otimes df(X, Y).
\eea
Therefore, 
\be 
\frac 12 {\mathcal L}_{\nabla f}\left( \frac {df}{|\nabla f|} \otimes \frac {df}{|\nabla f|} \right) 
= N(|\nabla f|)  \frac {df}{|\nabla f|} \otimes \frac {df}{|\nabla f|}.\label{eqnt2}
\ee
Since
$$ 
\frac 12 {\mathcal L}_{\nabla f} (|\nabla f|^2 g_{\Sigma})
= \frac 12 \nabla f(|\nabla f|^2) g_\Sigma
= Ddf(\nabla f, \nabla f)g_\Sigma= N(|\nabla f|) |\nabla f|^2 g_\Sigma,
$$
we conclude that
$$ 
\frac 12 {\mathcal L}_{\nabla f}\bar{g}=\bar{D}df= N(|\nabla f|) \bar{g}.
$$
In particular, we have $\bar{\tr} f =n N(|\nabla f|) $.
\end{proof}

\begin{lem} \label{lem2019-6-22-1}
 Let $(M^n, g, f)$ be a  closed generalized $(\lambda, n+m)$-Einstein 
manifold with constant scalar curvature $s$ and $\omega=0$.    If $f^{-1}(0)$ is connected, then $T=0$ on $M$.
\end{lem}
\begin{proof}
Let $p, q \in M$ be the only two points such that $f(p) = \min_M f$ and $f(q) = \max_M f$, respectively, and let $\bar M = M \setminus \{p, q\}$.
 Due to Lemma~\ref{lem2021-5-15-16}  together with our assumption that $f^{-1}(0)$ is connected and Lemma~\ref{lemt1}, we can apply Tashiro's result \cite{tas} and can see that
  $(\bar M, \bar g)$  is conformally equivalent
to ${\Bbb S}^n\setminus \{\bar p, \bar q\}$, where $\bar p$ and $\bar q$ are the points in ${\Bbb S}^n$  corresponding to $p$ and $q$, respectively.  
In particular, by Theorem 1 in \cite{bgv}, 
the fiber space $(\Gamma, g|_{\Gamma})$ is a space of constant curvature.
Thus, 
$$
(\Gamma, g|_\Gamma) \equiv ({\Bbb S}^{n-1}, \, r \cdot g_{{\Bbb S}^{n-1}}),
$$
where $r>0$ is a positive constant and $g_{{\Bbb S}^{n-1}}$ is a round metric.

Now, replacing $\Gamma = f^{-1}(0)$ by $\Gamma_t:= f^{-1}(t)$ in  (\ref{eqn2019-8-27-1}),
 it can be easily concluded that the warped product metric $\bar g_t$ also
satisfies Lemma~\ref{lemt1}, and hence, the same argument mentioned above 
shows that,
for any level hypersurface $\Gamma_t:= f^{-1}(t)$,
$$
(\Gamma_t, g|_{\Gamma_t}) \equiv ({\Bbb S}^{n-1}, r(t) \cdot  g_{{\Bbb S}^{n-1}}).
$$
Therefore, the original metric $g$ can also be written as
\be
g= \frac {df}{|\nabla f|}\otimes \frac {df}{|\nabla f|}+ b(f)^2 g_\Gamma,\label{eqn2019-12-23-1}
\ee
where $b (f)>0$ is a positive function depending only on $f$. 
From (\ref{eqnt2}) and the following identity
$$ 
\frac 12 {\mathcal L}_{\nabla f}(b^2 g_\Gamma)=b\langle \nabla f, \nabla b\rangle g_\Gamma = b|\nabla f|^2 \frac {db}{df}g_\Gamma,
$$
we obtain
\be 
\frac 12 {\mathcal L}_{\nabla f}g=N(|\nabla f|)\frac {df}{|\nabla f|}\otimes \frac {df}{|\nabla f|}+b|\nabla f|^2 \frac{db}{df}g_\Gamma.\label{eqnt5-1}
\ee
On the other hand, from (\ref{eqn1}) and (\ref{eqn2019-12-23-1}), we have
\bea
\frac 12 {\mathcal L}_{\nabla f}g &=& 
 \frac{f}{m}({\rm Ric}-\lambda g)\\
&=& N(|\nabla f|) \frac {df}{|\nabla f|}\otimes \frac {df}{|\nabla f|} + \frac{f}{m}{\rm Ric}
- \frac{f}{m}{\rm Ric}(N, N)  \frac {df}{|\nabla f|}\otimes \frac {df}{|\nabla f|}  
- \frac{\lambda f}{m} b^2 g_\Gamma.
\eea
Comparing this to (\ref{eqnt5-1}), we obtain
\be
\left( b|\nabla f|^2 \frac {db}{df} +\frac{\lambda f}{m} b^2\right) g_\Gamma
= \frac{f}{m} \left({\rm Ric} - {\rm Ric}(N, N)  \frac {df}{|\nabla f|}\otimes \frac {df}{|\nabla f|}\right).\label{eqnt7-1}
\ee
Now, let $\{e_1, e_2, \cdots, e_n\}$ be a local frame with $e_1=N$. 
Then, by applying $(e_j, e_j)$ for each $2\leq j\leq n$ to (\ref{eqnt7-1}), we have
$$ 
b|\nabla f|^2 \frac {db}{df}=  \frac{f}{m} {\rm Ric}(e_j, e_j)-\frac {\lambda f}{m}b^2
$$
for $2\leq j\leq n$. Summing up these, we obtain
$$ 
(n-1)b|\nabla f|^2 \frac {db}{df}= \frac{f}{m}\left[s-{\rm Ric}(N, N) \right]
-\frac {(n-1)\lambda f}{m} b^2.
$$
Substituting this into (\ref{eqnt7-1}), we get
\be
 \frac {1}{n-1} \left[s-{\rm Ric}(N, N)\right] g_\Gamma
  = \left({\rm Ric} - {\rm Ric}(N, N)  \frac {df}{|\nabla f|}\otimes \frac {df}{|\nabla f|}\right). \label{eqn2020-8-29-1}
 \ee
Replacing $(\Gamma, g_\Gamma)$ by $(\Gamma_t , g_{\Gamma_t})$, we can see that the argument mentioned above is also valid. 
Thus, (\ref{eqn2020-8-29-1}) shows that, on each level hypersurface $f^{-1}(t)$, we have
$$ 
{\rm Ric}(e_j, e_j)= \frac {1}{n-1} \left[s-{\rm Ric}(N, N)\right]
$$
for $2\leq j\leq n$, which is equivalent to 
$$
z(e_i, e_i) = - \frac{\a}{n-1}.
$$
Hence,
$$ 
|z|^2=\a^2+\frac {\a^2}{n-1}=\frac n{n-1}\a^2=\frac n{n-1}|i_Nz|^2,
$$
since $z(N, e_i)=0$ for $i\geq 2$. As a result, we have $T=0.$

\end{proof}

\begin{thm}\label{str002} 
Let $(M, g, f)$ be a  closed generalized $(\lambda, n+m)$-Einstein 
manifold.
 Assume the scalar curvature $s$ is  constant and $\omega=0$ If  $f^{-1}(0)$ is connected,  then  $M$ is isometric to a sphere ${\Bbb S}^{n}$.
\end{thm} 
\begin{proof}
By Lemma~\ref{lem2021-5-27-3} together with (\ref{eqn2021-5-27-2}) and (\ref{eqn2021-5-27-1}),  we have
\be
{\rm Ric}(N, N)   =\frac{(n-1)\lambda-s}{m-1}>0\quad \mbox{and}
\quad
\a = \frac{(n-1)\lambda-s}{m-1} - \frac{s}{n} \le 0.\label{eqn2021-5-29-1}
\ee
Suppose that $\a <0$. From Lemma~\ref{lem2019-6-22-1} and (\ref{eqn2021-4-20-8}), 
we  have
$$ 
\frac{n}{n-1}\a^2 = |z|^2 = \a(s-n\lambda)
$$
and so
\bea
\a = \frac{n-1}{n}s- (n-1)\lambda = [s-(n-1)\lambda]-\frac{s}{n}.\label{eqn2021-5-29-2}
\eea
Comparing this to (\ref{eqn2021-5-29-1}), we have $(n-1)\lambda - s =0$, which
 contradicts ${\rm Ric}(N, N)> 0$.
 Hence we have $\a = 0$ and consequently $(M, g)$ is Einstein, and  is isometric to a sphere.

\end{proof}

\section{Final Remarks}

Let $(M^n, g, f)$ be a  closed generalized $(\lambda, n+m)$-Einstein 
manifold with $f^{-1}(0) \ne \emptyset$. If $\o = 0$ and the scalar curvature $s$ is constant,
by Lemma~\ref{lem2021-5-4-1},   the function $\lambda$ is also a (positive) constant
(see also Corollary~\ref{cor2021-4-19-12} for $m=1$ without vanishing of $\o$).
In case of $m \ge 1$, we can show that the converse is also true. Namely, we have the following.

\begin{lem}\label{constantscalar}
Let $(M, g, f)$ be a  closed generalized $(\lambda, n+m)$-Einstein 
manifold with $f^{-1}(0) \ne \emptyset$. If $\o = 0$ and the function $\lambda$ is constant,
then the scalar curvature $s$ is also a (positive) constant.
\end{lem}
\begin{proof}
Suppose that $\lambda$ is constant. We recall that, from Lemma~\ref{lem2021-5-15-6}, the function $\a$ is constant. Since $i_{\n f} z = \a df$, it follows from Lemma~\ref{lem2021-4-19-100} that
$$
ds + \frac{2(m+n-1)}{n} s  df = 2\left[(n-1)\lambda - (m-1)\a \right] df.
$$
Now let $k_1 = \frac{m+n-1}{n}$ and $k_2 = (n-1)\lambda - (m-1)\a$. 
Considering the set $M^0 = \{x\in M\,:\, f(x)>0\}$ and multiplying by $\frac{1}{f}$, we obtain, on the set $M^0$,
\be
\n s + 2k_1 s \n \ln f = 2k_2 \n \ln f.\label{eqn2021-8-4-1}
\ee
Defining $\varphi = 2k_1 \ln f$, we can rewrite (\ref{eqn2021-8-4-1}) as 
$$
\n s + s\n \vp =  \frac{k_2}{k_1} \n \vp,
$$
or equivalently
$$
\n (s e^\vp) = \frac{k_2}{k_1} \n e^\vp.
$$
Since $k_1$ and $k_2$ are constants, we conclude that
$$
s = \frac{k_2}{k_1} + c_0 e^{-\vp}=\frac{k_2}{k_1} + \frac{c_0}{f^{2k_1}}.
$$
Consequently, we obtain
$$
sf^{2k_1} = \frac{k_2}{k_1} f^{2k_1} + c_0
$$
on the set  $M^0 = \{x\in M\,:\, f(x)>0\}$. Taking a sequence $p_l \in M^0$ tending to a point
$p \in f^{-1}(0)$ as $l \to \infty$, we have $c_0 = 0$ and hence
\be
s = \frac{k_2}{k_1} = \frac{n\left[(n-1)\lambda - (m-1)\a\right]}{m+n-1}
\label{eqn2021-8-4-2}
\ee
which shows that $s$ is constant on the set $f>0$. The same argument works on the set
$M_0 := \{x\in M \,:\, f(x) <0\}$ and the proof is complete.
\end{proof}

Note that (\ref{eqn2021-8-4-2}) is exactly the same  as (\ref{eqn2021-5-4-2}) in
Lemma~\ref{lem2021-5-4-1}. Applying Theorem~\ref{str001}  and Theorem~\ref{str002},
we have the following result.

\begin{thm}\label{thm2021-8-4-3} 
Let $(M^n, g, f)$ be a  closed generalized $(\lambda, n+m)$-Einstein 
manifold.  Assume $(M, g)$ has PIC  and  the function $\lambda$ is  constant. Then 
\begin{itemize}
\item[(1)]  if $f^{-1}(0)$ is connected,  then  $M$ is isometric to a sphere ${\Bbb S}^{n}$.
\item[(2)] if $f^{-1}(0)$ is disconnected,  then, it has only two connected components and
 $M$ is isometric to ${\Bbb S}^1 \times {\Bbb S}^{n-1}$,
 up to finite cover and rescaling.
\end{itemize}
\end{thm}


\setlength{\baselineskip}{15pt}

\vspace{.12in}
\noindent
Seungsu Hwang\\
Department of Mathematics\\
Chung-Ang University\\ 
221 HeukSuk-dong, DongJak-ku, Seoul, Korea 156-756\\
{\tt E-mail:seungsu@cau.ac.kr}

\vspace{.12in}
\noindent
Marcio Santos\\
Departamento de Matemática\\
Universidade Federal da Paraíba\\
João Pessoa, PB, Brazil\\
{\tt E-mail:marcio.santos@academico.ufpb.br}

\vspace{.12in}
\noindent
Gabjin Yun\\
Department of Mathematics and the Natural Science of Research Institute\\
Myong Ji University\\
San 38-2, Nam-dong, Yongin-si, Gyeonggi-do, Korea, 17058\\
{\tt E-mail:gabjin@mju.ac.kr}

\end{document}